\documentclass[11pt,reqno]{amsart}
\usepackage[T1]{fontenc}
\usepackage[utf8]{inputenc}
\usepackage{lmodern}
\usepackage[letterpaper,hmargin=1.12in,vmargin=1in]{geometry}
\usepackage{amsmath,amssymb,amsthm,mathtools,stmaryrd}
\usepackage{float}
\usepackage{needspace}
\usepackage{tikz}
\usetikzlibrary{arrows.meta,calc}
\definecolor{ggblue}{RGB}{28,92,150}
\definecolor{gglight}{RGB}{224,239,248}
\definecolor{ggpurple}{RGB}{119,54,147}
\definecolor{ggred}{RGB}{181,39,51}
\definecolor{gggray}{RGB}{100,110,120}
\tikzset{ggaxis/.style={-{Stealth[length=1.8mm]},draw=black!80,line width=.45pt},ggline/.style={draw=ggblue,line width=1.15pt},ggarrow/.style={-{Stealth[length=2.1mm]},line width=.9pt},gglabel/.style={font=\small},ggsmall/.style={font=\footnotesize}}

\usepackage[expansion=false]{microtype}
\usepackage[colorlinks=true,linkcolor=blue,citecolor=blue,urlcolor=blue]{hyperref}
\hypersetup{
  pdftitle={A growth gap for anisotropic minimal graphs},
  pdfauthor={Yang Yang},
  pdfsubject={Gradient growth gap and Bernstein rigidity for anisotropic minimal graphs},
  pdfkeywords={Anisotropic minimal graph, growth gap, Bernstein theorem, gradient growth, Jacobi field}
}
\newtheorem{theorem}{Theorem}[section]
\newtheorem{lemma}[theorem]{Lemma}
\newtheorem{proposition}[theorem]{Proposition}
\theoremstyle{remark}
\newtheorem{remark}[theorem]{Remark}
\newcommand{\R}{\mathbb R}
\newcommand{\Z}{\mathbb Z}
\newcommand{\Sn}{\mathbb S}
\newcommand{\Hn}{\mathcal H}
\newcommand{\M}{\mathbf M}
\newcommand{\restr}{\mathbin{\llcorner}}
\newcommand{\current}[1]{\llbracket #1\rrbracket}
\DeclareMathOperator{\spt}{spt}
\DeclareMathOperator{\conv}{conv}
\DeclareMathOperator{\reg}{reg}
\DeclareMathOperator{\sing}{sing}
\DeclareMathOperator{\tr}{tr}
\DeclareMathOperator{\Tr}{Tr}
\DeclareMathOperator{\Div}{div}
\DeclareMathOperator*{\esssup}{ess\,sup}
\DeclareMathOperator*{\essinf}{ess\,inf}
\allowdisplaybreaks[1]
\title[A growth gap for anisotropic minimal graphs]{A growth gap for anisotropic minimal graphs}
\author{Yang Yang}
\address{School of Mathematics, Hunan University}
\email{yyang1@hnu.edu.cn}
\date{}

\begin{document}

\begin{abstract}
We prove a growth gap for the gradient of entire anisotropic minimal graphs.
For each $n\geq2$ and each smooth uniformly elliptic parametric integrand
$\Phi$ on $\R^{n+1}$, there is an exponent $\alpha=\alpha(n,\Phi)>0$
such that every smooth nonaffine entire $\Phi$-minimal graph
$u:\R^n\to\R$ satisfies $\sup_{B_R^n}|Du|\geq cR^\alpha$ for all
$R\geq R_0$, for some constants $c>0$ and $R_0<\infty$ depending on
the solution.
In particular, gradient growth $o(R^\alpha)$ forces flatness, resolving
a conjecture of Mooney and the author. The result holds in every dimension,
without any assumption that $\Phi$ is close to the Euclidean area integrand.
\end{abstract}

\maketitle

\section{Introduction}

The anisotropic Bernstein problem asks when an entire graph critical for an
elliptic parametric functional must be a hyperplane. For Euclidean area, the
answer depends only on the dimension: every entire minimal graph over $\R^n$
is affine for $n\leq 7$, whereas nonaffine examples exist for
$n\geq 8$ \cite{Almgren66,BDGG,DeGiorgi,Fleming,Simons}.
For a general anisotropic energy, Jenkins \cite{Jenkins} and Simon
\cite{Simon} proved rigidity in dimensions two and three. Mooney
\cite{Mooney} constructed nonflat examples in dimension six, and Mooney
and the author \cite{MYBernstein} subsequently constructed examples in
dimension four, settling the remaining dimensional question. Thus no
unrestricted Bernstein theorem can hold for all elliptic anisotropies once
$n\geq 4$.

The growth of the gradient gives a finer distinction between flat and nonflat
graphs. A bounded-gradient entire solution is affine for every fixed elliptic
integrand, by applying the De Giorgi--Nash--Moser oscillation estimate to its
derivatives. In the Euclidean case, Ecker and Huisken \cite{EH} proved
rigidity in every dimension under the much weaker assumption
$|Du(x)|=o(|x|)$. In their study of anisotropic foliations of Lawson cones,
Mooney and the author asked whether bounded-gradient rigidity always extends
to some positive power of the radius \cite[Conjecture 3.3]{MYLawson}.
The exponent was allowed to depend on the integrand. This qualification is
natural: the constructions in \cite{MYLawson,MYBernstein} vary the integrand
together with the growth rate, and do not provide arbitrarily slow growth
for a single fixed anisotropy.

We give an affirmative answer to this question. Let $n\geq2$, and let
$\Phi:\R^{n+1}\setminus\{0\}\to(0,\infty)$ be smooth and positively
one-homogeneous. We assume that there are constants
$0<\lambda\leq\Lambda<\infty$ such that
\begin{equation}\label{ellipticity}
\lambda|\tau|^2\leq D^2\Phi(\nu)[\tau,\tau]\leq\Lambda|\tau|^2,
\qquad \nu\in\Sn^n,\quad\tau\perp\nu.
\end{equation}
With $\Phi(0)=0$, the above assumptions in particular
make $\Phi$ convex on $\R^{n+1}$. For an oriented hypersurface $\Sigma$
with chosen unit normal $\nu_\Sigma$, write $\Hn^k$ for $k$-dimensional
Hausdorff measure ($k\geq0$). The associated energy is
\[
\mathcal A_\Phi(\Sigma)=\int_\Sigma\Phi(\nu_\Sigma)\,d\Hn^n.
\]
For $p\in\R^n$ set $\varphi(p)=\Phi(-p,1)$. The Euler--Lagrange
equation for a graph $\Sigma=\{(x,u(x)):x\in\R^n\}$ with the upward
orientation $\nu_\Sigma\cdot e_{n+1}>0$ is
\begin{equation}\label{graph-equation}
\Div D\varphi(Du)=0.
\end{equation}
We write $B_R^n(x)$ and $B_R(X)$ for open balls in the base $\R^n$
and the ambient space $\R^{n+1}$, respectively; an omitted center is zero.

\begin{theorem}\label{main}
Let $n\geq 2$, and let $\Phi:\R^{n+1}\setminus\{0\}\to(0,\infty)$ be
smooth, positively one-homogeneous, and uniformly elliptic in the sense of
\eqref{ellipticity}. There exists $\alpha=\alpha(n,\Phi)>0$ such that every
smooth entire solution $u:\R^n\to\R$ of \eqref{graph-equation} with
\begin{equation}\label{small-growth}
\sup_{B_R^n}|Du|=o(R^\alpha)\qquad\text{as }R\to\infty
\end{equation}
is affine. Moreover, for every nonaffine entire solution there are
constants $c>0$ and $R_0<\infty$, depending on the solution, such that
\begin{equation}\label{lower-growth}
\sup_{B_R^n}|Du|\geq cR^\alpha\qquad(R\geq R_0).
\end{equation}
\end{theorem}

The integrand need not be even. It is fixed throughout the theorem, and the
exponent is independent of the solution. Setting $\epsilon=\alpha/2$
also gives the formulation in \cite[Conjecture 3.3]{MYLawson}, which uses
$O(R^\epsilon)$ instead of $o(R^\alpha)$.
For $n\geq 3$, the stronger assertion \eqref{lower-growth} follows from a
fixed-scale estimate, Proposition~\ref{fixed-scale}: once the reciprocal
vertical Jacobi field is sufficiently large relative to its value at a
center, its maximum must double when the radius is multiplied by a fixed
number $L=L(n,\Phi)>2$. One can take $\alpha=\log 2/\log L$, although the
proof does not give an effective value of $L$. In dimensions two and three,
unrestricted rigidity is already known by the theorems of Jenkins
\cite{Jenkins} and Simon \cite{Simon}, respectively. The nontrivial range
for the growth-gap conclusion is therefore $n\geq 4$.

Earlier controlled-growth results imposed a perturbative assumption on the
energy. Simon \cite{Simon} established a Bernstein theorem in the classical
low-dimensional range for integrands sufficiently close to area. Du and the
author \cite{DY} proved an all-dimensional result for integrands sufficiently
$C^3$-close to area, with an admissible gradient-growth exponent tending to
one as the integrand tends to area. That argument uses an anisotropic
Simons inequality. Theorem~\ref{main} removes the closeness assumption, but
does not retain this quantitative information about the exponent. In
particular, it does not assert the Euclidean endpoint
$\sup_{B_R^n}|Du|=o(R)$ for a general anisotropy.

The positive vertical translation Jacobi field is available on every graph,
and its reciprocal is a subsolution of a tangentially uniformly elliptic
equation. However, a rescaled sequence of graphs can converge to a singular
minimizing boundary, and a point at which the normalized reciprocal is large
may converge to its singular set. Moreover, one cannot assume that this
limit is a cone: the usual Euclidean monotonicity formula is unavailable
for a general parametric integrand \cite{Allard}.

Our argument uses a mean-value estimate for positive truncations to retain
more than a nonzero limiting function. Failure of a fixed-scale growth
estimate would give rescalings for which the normalized reciprocals have
maximum one in $\overline B_1$, tend to zero at the center, and have upper
bounds tending to one on every fixed ball. The mean-value estimate implies
that their limit equals one on a set of positive $\Hn^n$ measure. A local
strong maximum principle then makes the limit constant on a regular
component. The nonnegative curvature term in the Jacobi equation forces
that component to be planar. This compactness step is isolated in
Proposition~\ref{maximum-compactness}.

To pass from a planar component to a single hyperplane, we use wall rigidity
for global anisotropic perimeter minimizers
\cite[Proposition 1.6]{YangHalfspace}. We include a proof of this input from
the wall-contact principle of De Philippis and Maggi \cite{DPM}. The passage
requires some care when the integrand is not even: restriction to the two
sides of the plane must preserve ambient minimality, with the orientation
reversed on the complementary phase. We give this argument in
Section~\ref{sec-planar}. The geometric input is distinct from the halfspace
graph theorem of Du, Mooney, the author, and Zhu \cite{DMYZ}, which concerns
graphs with affine boundary data; the limit here is not initially known to
be a graph. Once it is a plane, a quotient of two translation Jacobi fields
gives the final Harnack comparison, even if that plane is vertical.

The paper is organized as follows. In Section~\ref{sec-estimates}, we collect
the variational and elliptic estimates used in the proof.
Section~\ref{sec-planar} establishes the planar-component rigidity result,
and Section~\ref{sec-maximum} proves the maximum-set compactness proposition.
Finally, in Section~\ref{sec-growth} we derive the fixed-scale growth estimate
and prove Theorem~\ref{main}.

\section*{Acknowledgments}
The author is grateful to Connor Mooney for suggesting this problem and for
inspiring discussions.

\section{Variational and elliptic estimates}\label{sec-estimates}

In this section we fix the orientation conventions and collect the
estimates used in the compactness argument. Absolute minimality gives a
Sobolev inequality for compactly supported functions, from which we derive
a mean-value estimate for subsolutions. No global Neumann Poincar\'e
inequality is required. The constants in these estimates depend only on
the dimension and the fixed integrand.

For a set $E$ of locally finite perimeter, $D\chi_E$ is the
distributional derivative of its characteristic function and $|D\chi_E|$
its total variation measure. Let $\partial^*E$ denote the reduced
boundary and $\nu_E$ its measure-theoretic outer unit normal, so that
$-D\chi_E=\nu_E|D\chi_E|$. We use $\spt$ for support and
$\sigma\restr A$ for the restriction of a scalar or vector measure
$\sigma$ to a Borel set $A$. For a bounded open set
$W\subset\R^{n+1}$, define
\[
P_\Phi(E;W)=\int_{W\cap\partial^*E}\Phi(\nu_E)\,d\Hn^n.
\]
We call $E$ a global $\Phi$-perimeter minimizer if
$P_\Phi(E;W)\leq P_\Phi(F;W)$ whenever $W$ is a bounded open set, $F$ has
locally finite perimeter, and $E\mathbin{\triangle}F\subset K$ for some
compact set $K\subset W$. Inclusions and equalities between sets of finite
perimeter are understood up to Lebesgue null sets. We write
\[
S_E=\spt|D\chi_E|,\qquad
\mu_E=|D\chi_E|=\Hn^n\restr\partial^*E,
\]
where the last identity is the standard representation of the perimeter
measure \cite{Federer}. A point of $S_E$ is regular if the support is a
smooth embedded hypersurface near that point. Write $\reg S_E$ for the
set of such points and $\sing S_E=S_E\setminus\reg S_E$ for the singular
set. For the minimizing boundaries considered below, the regularity estimate \eqref{singular-size} implies that the singular
set has zero $\Hn^n$ measure. Since every regular boundary point belongs
to $\partial^*E$, the set $S_E\setminus\partial^*E$ is $\Hn^n$-null.
Consequently, $\mu_E=\Hn^n\restr S_E$ for these boundaries.

We equip $\R^{n+1}$ with its standard orientation. For
$g\in BV_{\mathrm{loc}}(\R^{n+1};\Z)$, let $\current{g}$ denote the locally
integral $(n+1)$-current defined by
\[
\current{g}(\eta)=\int_{\R^{n+1}}g\,\eta
\]
for every smooth compactly supported $(n+1)$-form $\eta$. For a set $E$
of locally finite perimeter, write $\current{E}=\current{\chi_E}$.
The boundary operator on currents is defined by
$(\partial T)(\omega)=T(d\omega)$, where $d$ is the exterior derivative
and $\omega$ is a smooth compactly supported form of the appropriate
degree. We set
\[
T_E:=\partial\current{E},
\]
the locally integral $n$-current associated with the oriented boundary
of $E$; see \cite{Federer}.

An integral current $Z$ is called a cycle when $\partial Z=0$.
For a locally integral current $T$ of any dimension, write $\|T\|$ for
its mass measure and $\M(T;W)=\|T\|(W)$ for its mass in $W$; an omitted
$W$ means the whole ambient space. For an $n$-current $T$ in
$\R^{n+1}$, let $\mathbf n(T)$ denote its oriented normal vector measure.
For $T_E$ we have $\|T_E\|=\mu_E$, and our sign convention is
$\mathbf n(T_E)=-D\chi_E$. The quotient
$d\mathbf n(T)/d\|T\|$ denotes the Radon--Nikodym density of the normal
vector measure with respect to the mass measure. For a Borel set $W$,
the oriented energy of $T$ is
\[
\M_\Phi(T;W)=\int_W\Phi\left(\frac{d\mathbf n(T)}{d\|T\|}\right)\,d\|T\|.
\]
We call $T$ absolutely minimizing for this energy if
$\M_\Phi(T;W)\leq\M_\Phi(T+Z;W)$ for every bounded open set $W$ and
every compactly supported integral $n$-cycle $Z$ with $\spt Z\Subset W$.
Choose constants $a,b>0$ such that $a\leq\Phi\leq b$ on $\Sn^n$.
Convexity and homogeneity give subadditivity of the oriented energy,
and these bounds make it comparable with ordinary mass.
Define the reflected integrand by $\Phi^-(z)=\Phi(-z)$ for
$z\in\R^{n+1}$. When $\Phi$ is not even, reversing a current replaces
$\Phi$ by $\Phi^-$.

\begin{lemma}\label{current-minimality}
If $E$ is a global $\Phi$-perimeter minimizer, then $T_E$ is absolutely
minimizing for the oriented energy $\M_\Phi$.
\end{lemma}

\begin{proof}
We recall the argument of \cite[Proposition 5.4]{YangHalfspace}. Let $Z$ be
a compactly supported integral $n$-cycle and choose a ball $B$ with
$\conv(\spt Z)\subset B$. The cone construction and the representation of
top-dimensional integral currents give $Z=\partial\current{g}$ for a
compactly supported $g\in BV(\R^{n+1};\Z)$ with $\spt g\subset B$
\cite{Federer}. Set $\vartheta=\chi_E+g$ and
$G=\{\vartheta\geq 1\}$. The monotone truncation
$\tau(t)=\min\{1,\max\{0,t\}\}$ satisfies $\tau(\vartheta)=\chi_G$.
The scalar BV chain rule expresses $D(\tau(\vartheta))$ as a nonnegative
scalar multiple, at most one, of $D\vartheta$. Thus it does not increase the oriented
$\Phi$-variation. Since $G=E$ outside $\spt g$, set minimality gives
\[
\M_\Phi(T_E;B)\leq\M_\Phi(T_G;B)
\leq\M_\Phi(\partial\current{\vartheta};B)=\M_\Phi(T_E+Z;B).
\]
The comparison is local, since the two currents agree outside the support
of $Z$.
\end{proof}

An entire solution of \eqref{graph-equation} produces such a minimizer.
Indeed, for its subgraph $E=\{(x,t):t<u(x)\}$ the outer normal is
$\nu=(-Du,1)/(1+|Du|^2)^{1/2}$. The vector field
\[
Z(x,t)=D\Phi(\nu(x))
\]
is independent of $t$ and divergence-free by \eqref{graph-equation}.
Convexity and Euler's identity give $Z\cdot\xi\leq\Phi(\xi)$ for every
$\xi\in\R^{n+1}$, with equality at $\xi=\nu$. The associated closed $n$-form
calibrates $\partial\current{E}$, proving absolute minimality directly.
Translations and positive dilations preserve this property and the
integrand.

We will use the following consequences of minimizing-boundary compactness
and regularity:
\begin{equation}\label{density}
cr^n\leq\mu_E(B_r(X))\leq Cr^n\qquad(X\in S_E,\ r>0).
\end{equation}
If $E_j$ are global minimizers and $0\in S_{E_j}$, then, after passage to
a subsequence,
\begin{equation}\label{compactness}
\chi_{E_j}\to\chi_E\text{ in }L^1_{\mathrm{loc}},\qquad
D\chi_{E_j}\stackrel{*}{\rightharpoonup}D\chi_E,\qquad
\mu_{E_j}\stackrel{*}{\rightharpoonup}\mu_E,
\end{equation}
where $E$ is a global minimizer, $0\in S_E$, and the supports converge
locally in Hausdorff distance. The density bounds and the full-space form
of this compactness statement follow from \cite[Lemma 2.8 and Theorem
2.9]{DPM}. The convergence of the unweighted perimeter measures in
\eqref{compactness} is part of the minimizing compactness theorem, not
merely BV lower semicontinuity.

The singular set $\sing S_E$, defined above, is closed.
The regularity theory of Schoen, Simon, and Almgren \cite{SSA}
gives, in particular, the estimate
\begin{equation}\label{singular-size}
\Hn^{n-1}(\sing S_E)=0.
\end{equation}
For a precise sufficient statement one may use
\cite[Corollary 2.5]{Figalli}. Lemma~\ref{current-minimality} supplies the
oriented absolute minimality required there, and boundary currents have
multiplicity one. Our smooth, positively one-homogeneous integrand
satisfies the convexity and ellipticity assumptions of that result after
multiplication by a positive constant; evenness is not required.

On compact subsets of $\reg S_E$, convergence in \eqref{compactness} is
one-sheet graphical and smooth, with the outer normals converging as well.
We spell out why no additional sheets can be lost in this assertion.
Write $\mu_j=\mu_{E_j}$ and $\nu_j=\nu_{E_j}$ for the perimeter
measure and outer unit normal of $E_j$. For a fixed unit vector $N$
and a nonnegative $\eta\in C_c(\R^{n+1})$, the orientation convention gives
\[
\int\eta|\nu_{E_j}-N|^2\,d\mu_{E_j}
=2\int\eta\,d\mu_{E_j}+2\int\eta\,N\cdot dD\chi_{E_j}.
\]
Both terms on the right converge by \eqref{compactness}; hence the
weighted oriented excess converges to that of $E$. Near a regular point
of $S_E$, first choose a small cylinder in which the limiting boundary is
a graph with small slope over its tangent plane, oriented by $N$.
The preceding identity, with cutoffs supported in a slightly larger
cylinder, transfers small excess to $E_j$. Support convergence excludes
boundary points near the top and bottom faces; local $L^1$ convergence
then fixes the two phases there. Consequently the boundary current has
projection of multiplicity one onto the base disk. The flat regularity
theorem \cite[Theorem 2.1]{Figalli} applies and represents the whole
support in a smaller cylinder as a single $C^{1,\gamma}$ graph for some
$\gamma\in(0,1)$, with uniform local estimates. Local elliptic estimates for the fixed smooth
integrand upgrade this to smooth convergence. A finite covering gives
the assertion on each compact regular subset. Neither this assertion nor
\eqref{singular-size} assumes that $S_E$ is a cone.

Let now $\Sigma=S_E$ be smooth, with outer unit normal $\nu=\nu_E$
and induced measure $\mu=\Hn^n\restr\Sigma$. Let $A$ be its second
fundamental form, viewed as a self-adjoint endomorphism of $T\Sigma$,
and write $\nabla_\Sigma$ and $\Div_\Sigma$ for the tangential gradient
and divergence. Put
\begin{equation}\label{jacobi-operator}
P=D^2\Phi(\nu)|_{T\Sigma},\qquad
L=\Div_\Sigma(P\nabla_\Sigma),\qquad Q=\tr(PA^2).
\end{equation}
In particular, $Q\geq\lambda|A|^2$. The same notation is used on
regular portions of minimizing boundaries. For a sequence of smooth
hypersurfaces $\Sigma_j$ or a minimizing support $S$, subscripts $j$
and $S$ on $A,P,L,Q,\nu,\mu$ indicate the corresponding objects on
$\Sigma_j$ and $\reg S$, respectively. We abbreviate $\nabla_\Sigma$ to $\nabla$
when the hypersurface is clear. Translation invariance gives, for
every constant vector $e\in\R^{n+1}$,
\begin{equation}\label{translation-jacobi}
L(\nu\cdot e)+Q(\nu\cdot e)=0.
\end{equation}
This is the divergence-form Jacobi identity; it follows by differentiating
the first variation under translations, or from the anisotropic second
variation formula \cite[Section 2]{DY}. On an upward graph,
\[
f=\nu\cdot e_{n+1}=(1+|Du|^2)^{-1/2}>0,\qquad v=f^{-1}.
\]
For a tangent vector $\xi$, set $|\xi|_P^2=\langle P\xi,\xi\rangle$.
The chain rule yields the identity
\begin{equation}\label{reciprocal-jacobi}
Lv=Qv+2v^{-1}|\nabla_\Sigma v|_P^2\geq Qv\geq 0.
\end{equation}
Differential inequalities are understood weakly; our sign convention
is that $Lz\geq0$ is the subsolution inequality.

\begin{lemma}[Mean-value estimate]\label{mean-value}
Suppose $n\geq 3$ and $\Sigma$ is a smooth global $\Phi$-minimizing
boundary. There are positive constants $C_S,C_M$, depending only on $n,\Phi$,
such that, for every $\zeta\in C_c^1(\Sigma)$,
\begin{equation}\label{sobolev}
\left(\int_\Sigma|\zeta|^{2n/(n-2)}\,d\mu\right)^{(n-2)/n}
\leq C_S\int_\Sigma|\nabla_\Sigma\zeta|^2\,d\mu.
\end{equation}
For every $X\in\R^{n+1}$ and $r>0$, every nonnegative weak subsolution $Lz\geq 0$ in $\Sigma\cap B_{2r}(X)$
satisfies
\begin{equation}\label{mean-value-estimate}
\esssup_{\Sigma\cap B_r(X)}z
\leq C_Mr^{-n/2}
\left(\int_{\Sigma\cap B_{2r}(X)}z^2\,d\mu\right)^{1/2}.
\end{equation}
The constant is unchanged when $z$ is replaced by a positive truncation
of a subsolution.
\end{lemma}

\begin{proof}
For a relatively compact smooth domain $U\subset\Sigma$, let
$\partial U$ be its boundary in $\Sigma$ and $T_U$ its integration
current with the orientation induced from $\Sigma$. The isoperimetric theorem for integral currents
\cite{FF}, in the form \cite[4.2.10]{Federer}, gives a compactly supported
filling $S$ of $\partial T_U$ with
\[
\M(S)\leq C(n)\M(\partial T_U)^{n/(n-1)}.
\]
Absolute minimality passes to the oriented subcurrent $T_U$: one compares
$T_E$ with $T_E-T_U+S$ and cancels the unchanged, finite residual energy
in a large ball. Consequently
\[
a\,\Hn^n(U)\leq\M_\Phi(T_U)\leq\M_\Phi(S)
\leq C(n)b\,\Hn^{n-1}(\partial U)^{n/(n-1)}.
\]
The coarea formula gives the $W^{1,1}$ Sobolev inequality. Applying it to
$|\zeta|^{2(n-1)/(n-2)}$ and using Cauchy--Schwarz proves
\eqref{sobolev}.

For the second assertion, choose a nonnegative smooth ambient cutoff
$\eta$ supported in $B_{2r}(X)$ and an exponent $p\geq2$.
Test the weak subsolution inequality with $\eta^2z^{p-1}$, using
positive regularization and truncation when necessary. Ellipticity gives
\[
\int_\Sigma|\nabla_\Sigma(\eta z^{p/2})|^2\,d\mu
\leq Cp^2\int_\Sigma z^p|\nabla_\Sigma\eta|^2\,d\mu.
\]
Combine this with \eqref{sobolev} and iterate with exponents
$2(n/(n-2))^\ell$, $\ell=0,1,2,\ldots$, on nested ambient balls. The usual Moser iteration
\cite[Chapter 8]{GT} proves \eqref{mean-value-estimate}.
Ambient cutoffs have compact support on the closed embedded hypersurface
and satisfy $|\nabla_\Sigma\eta|\leq|D\eta|$. No connectedness assumption
on the intersection with a ball is needed. Finally, for any real threshold $k$, $(z-k)_+$ is a weak
subsolution whenever $z$ is one, as follows by convex approximation.
\end{proof}

\Needspace{8.8cm}
Figure~\ref{fig:no-spikes} illustrates the mean-value estimate in
Lemma~\ref{mean-value}. At fixed $r>0$, an interior essential supremum
of one forces the indicated lower bound on the integral over
$\Sigma\cap B_{2r}(X)$. Thus unit-height peaks with vanishing total
$L^2$ mass are excluded. The blue curve is a schematic spatial profile,
not an explicit subsolution or a time-dependent graph. This statement
concerns $n$-dimensional integral mass, not the width of a peak along
every one-dimensional slice; it also applies to positive truncations.
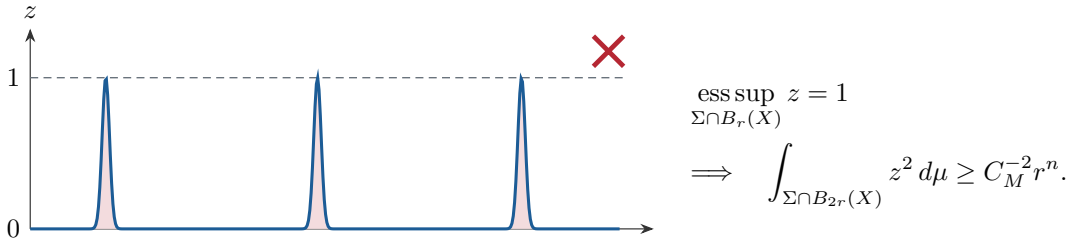
\begin{figure}[H]
\centering
\begin{tikzpicture}[x=1cm,y=1cm,font=\small]
\path[use as bounding box] (0,0) rectangle (15.65,3.8);
\begin{scope}[shift={(.55,.65)}]
  \draw[ggaxis] (0,0)--(8.25,0);
  \draw[ggaxis] (0,0)--(0,2.65) node[above] {$z$};
  \node[left] at (0,0) {$0$};
  \draw[gggray,densely dashed,line width=.5pt] (0,2.0)--(7.9,2.0);
  \node[left] at (0,2.0) {$1$};
  \fill[ggred!16] (0,0) -- plot[domain=0:7.8,samples=450]
   (\x,{2*exp(-((\x-1.0)/.070)^2)+2*exp(-((\x-3.8)/.070)^2)+2*exp(-((\x-6.5)/.070)^2)}) -- (7.8,0)--cycle;
  \draw[ggline] plot[domain=0:7.8,samples=450]
   (\x,{2*exp(-((\x-1.0)/.070)^2)+2*exp(-((\x-3.8)/.070)^2)+2*exp(-((\x-6.5)/.070)^2)});
  \draw[ggred,line width=1.8pt] (7.45,2.55)--(7.85,2.15) (7.45,2.15)--(7.85,2.55);
\end{scope}
\node[anchor=west] at (9.15,2.25) {$\displaystyle\esssup_{\Sigma\cap B_r(X)}z=1$};
\node[anchor=west] at (9.15,1.35) {$\displaystyle\Longrightarrow\quad\int_{\Sigma\cap B_{2r}(X)}z^2\,d\mu\geq C_M^{-2}r^n.$};
\end{tikzpicture}
\space
\caption{No spikes.}
\label{fig:no-spikes}
\end{figure}

\section{Planar components}\label{sec-planar}

In this section we show that a planar regular component of a global
minimizing boundary determines the whole boundary. We first prove the
wall-rigidity statement needed for this purpose. The small singular set
and an oriented restriction argument then reduce the planar-component
assertion to that statement.

The following result is \cite[Proposition 1.6]{YangHalfspace}. We recall
a proof based on the wall-contact principle \cite[Lemma 2.13]{DPM}, so
that its use below does not require a further rigidity input.

\begin{lemma}[Wall rigidity]\label{wall-rigidity}
Let $H$ be an open halfspace and let $F$ be a global $\Phi$-perimeter
minimizer. Then
\begin{equation}\label{wall}
F\subset H,\quad \spt|D\chi_F|\cap\partial H\ne\varnothing
\quad\Longrightarrow\quad F=H.
\end{equation}
\end{lemma}

\begin{proof}
Translate a contact point to zero and put $P=\partial H$.
Write $\Tr_P$ for the one-sided BV trace on $P$ taken from $H$.
The wall-contact principle of De Philippis and Maggi states that a minimizer
contained in $H$ and touching $P$ at zero has full one-sided BV trace on
a smaller wall disk. Its hypothesis is ambient minimality, with
competitors allowed to cross $P$, and it applies without an evenness
assumption. Applying it in arbitrarily large balls gives
$\Tr_P\chi_F=1$ almost everywhere on $P$.

Let $K=H\setminus F$. Then $\Tr_P\chi_K=0$, and the BV zero-extension
formula gives $|D\chi_K|(P)=0$. Since $\chi_F+\chi_K=\chi_H$,
\[
T_F=T_H-T_K.
\]
The mass measures of $T_H$ and $T_K$ are mutually singular, so this
splitting is additive for oriented energy. For any compactly supported
integral cycle $Z$, choose a ball $B$ with $\conv(\spt Z)\subset B$.
Lemma~\ref{current-minimality} and subadditivity give
\begin{align*}
\M_\Phi(T_H;B)+\M_\Phi(-T_K;B)
&=\M_\Phi(T_F;B)\\
&\leq\M_\Phi(T_F+Z;B)\\
&\leq\M_\Phi(T_H;B)+\M_\Phi(-T_K+Z;B).
\end{align*}
Cancel the finite plane term. Thus $-T_K$ is absolutely
$\Phi$-minimizing, and $K$ is a global $\Phi^-$-perimeter minimizer.

Suppose $K$ has positive measure. Its boundary support is nonempty,
since otherwise $\chi_K$ would be constant almost everywhere in
$\R^{n+1}$, contrary to $K\subset H$ and $|K|>0$. Fix
$X\in\spt|D\chi_K|$ and let $R_j\to\infty$. Apply minimizing
compactness to $F_j=R_j^{-1}F$ and $K_j=R_j^{-1}K$, taking a common
subsequence. The limits $F_\infty,K_\infty\subset H$ minimize
$\Phi,\Phi^-$, respectively, and
\begin{equation}\label{phase-identity}
\chi_{F_\infty}+\chi_{K_\infty}=\chi_H
\qquad\text{almost everywhere},
\end{equation}
by local $L^1$ convergence. Support persistence gives
$0\in\spt|D\chi_{F_\infty}|$, since zero belongs to every
$\spt|D\chi_{F_j}|$. It also gives $0\in\spt|D\chi_{K_\infty}|$,
since $X/R_j\in\spt|D\chi_{K_j}|$ and $X/R_j\to 0$. Applying the
wall-contact principle separately to the two limits shows that both have
trace one on $P\cap B_1$. But the one-sided BV trace of
\eqref{phase-identity} gives
\[
\Tr_P\chi_{F_\infty}+\Tr_P\chi_{K_\infty}=1
\qquad\text{almost everywhere on }P.
\]
This is a contradiction. Here linearity of the trace is used only for
the limiting identity; no convergence of the traces of $F_j$ or $K_j$
is asserted. Hence $K$ is null and $F=H$.
\end{proof}

\Needspace{9.4cm}
Figure~\ref{fig:wall} depicts Lemma~\ref{wall-rigidity} in a schematic
section, with $H$ above $P=\partial H$ and $F$ shaded more darkly.
In panel (a), $F\subset H$ touches $P$ at zero without filling $H$;
this configuration is excluded for a global $\Phi$-perimeter minimizer.
Panel (b) shows the resulting configuration $F=H$, up to null sets.
The global comparison class is essential: competitors may cross $P$,
so the wall is not an obstacle. The frames only truncate the view.
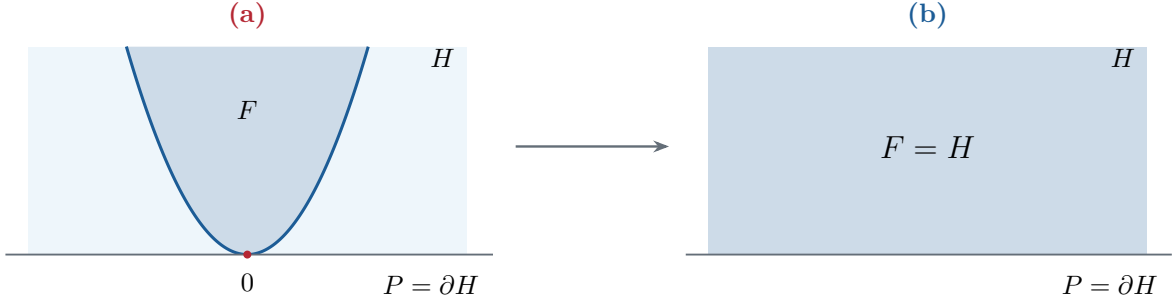
\begin{figure}[H]
\centering
\begin{tikzpicture}[x=1cm,y=1cm,font=\small]
\path[use as bounding box] (0,0) rectangle (15.65,4.1);
\node[font=\small\bfseries,text=ggred] at (3.25,3.87) {(a)};
\node[font=\small\bfseries,text=ggblue] at (12.25,3.87) {(b)};
\fill[gglight!55] (.35,.7) rectangle (6.15,3.45);
\begin{scope}
\clip (.35,.7) rectangle (6.15,3.45);
\fill[ggblue!22] (1.65,3.45) -- plot[domain=-1.6:1.6,samples=100] ({3.25+\x},{.7+1.075*\x*\x}) -- (4.85,3.45)--cycle;
\draw[ggline] plot[domain=-1.6:1.6,samples=100] ({3.25+\x},{.7+1.075*\x*\x});
\end{scope}
\draw[gggray,line width=.7pt] (.05,.7)--(6.5,.7);
\fill[ggred] (3.25,.7) circle (1.5pt);
\node[anchor=south east] at (6.13,3.05) {$H$};
\node at (3.25,2.61) {$F$};
\node[anchor=north] at (3.25,.57) {$0$};
\node[anchor=north east] at (6.45,.56) {$P=\partial H$};
\draw[ggarrow,gggray] (6.8,2.13)--(8.83,2.13);
\fill[ggblue!22] (9.35,.7) rectangle (15.15,3.45);
\draw[gggray,line width=.7pt] (9.05,.7)--(15.48,.7);
\node[anchor=south east] at (15.13,3.05) {$H$};
\node[font=\large] at (12.25,2.12) {$F=H$};
\node[anchor=north east] at (15.43,.56) {$P=\partial H$};
\end{tikzpicture}
\space
\caption{Wall rigidity.}
\label{fig:wall}
\end{figure}

\begin{lemma}[Planar component]\label{planar-component}
Let $E$ be a global $\Phi$-perimeter minimizer, and suppose that a
connected component $U$ of $\reg S_E$ has $A\equiv 0$. Then $E$ is a
halfspace.
\end{lemma}

\begin{proof}
The connected hypersurface $U$ is an open subset of an affine $n$-plane
$\Pi$, with constant outer normal. Write $\partial_\Pi U$ for the
topological boundary of $U$ relative to $\Pi$. We claim that
\begin{equation}\label{relative-boundary}
\partial_\Pi U\subset\sing S_E.
\end{equation}
A point in $\partial_\Pi U$ belongs to $S_E$ since the support is
closed. If it were regular, it would lie in the closure of $U$ inside
the manifold $\reg S_E$. A connected component is closed in that
manifold, so the point would belong to $U$. This contradicts the fact
that $U$ is open in $\Pi$.

Set $K=\Pi\cap\sing S_E$. It is closed, and $\Hn^{n-1}(K)=0$ by
\eqref{singular-size}. Such a set does not disconnect an $n$-plane.
For completeness, fix $a,b\in\Pi\setminus K$.
From either $a$ or $b$, the directions of segments meeting $K$ form a
set of zero $(n-1)$-dimensional spherical measure: apply radial
projection to countably many compact pieces separated from the point.
Almost every intermediate point $q\in\Pi$, with respect to
$\Hn^n\restr\Pi$, therefore gives a path $[a,q]\cup[q,b]$ avoiding $K$. By \eqref{relative-boundary}, $U$ is
both open and closed in $\Pi\setminus K$. Hence
\begin{equation}\label{whole-plane}
U=\Pi\setminus K,\qquad\Pi\subset S_E.
\end{equation}

Let $H^+$ and $H^-$ be the two open halfspaces bounded by $\Pi$,
labeled so that $\chi_E$ has trace one from $H^+$ and trace zero from
$H^-$ along the regular part of $\Pi$. By \eqref{whole-plane} these
traces hold almost everywhere on $\Pi$. Set $F=E\cap H^+$. The BV
zero-extension formula then gives the vector-measure identity
\begin{equation}\label{restriction}
D\chi_F=D\chi_E\restr(H^+\cup\Pi).
\end{equation}
Indeed, inside $H^+$ the derivatives agree, outside its closure both
sides vanish, and the trace jump on $\Pi$ is the same. Thus $T_F$ is an
oriented restriction of $T_E$, and its mass measure is mutually singular
with that of $T_R=T_E-T_F$. In particular,
\begin{equation}\label{additive-splitting}
\M_\Phi(T_E;B)=\M_\Phi(T_F;B)+\M_\Phi(T_R;B)
\end{equation}
for every bounded Borel set $B$. Both $T_F$ and $T_R$ are locally
integral cycles.

Let $Z$ be a compactly supported integral cycle, and choose a ball $B$
with $\conv(\spt Z)\subset B$. Lemma~\ref{current-minimality},
\eqref{additive-splitting}, and subadditivity imply
\begin{align*}
\M_\Phi(T_F;B)+\M_\Phi(T_R;B)
&\leq\M_\Phi(T_E+Z;B)\\
&\leq\M_\Phi(T_F+Z;B)+\M_\Phi(T_R;B).
\end{align*}
All terms are finite. Cancelling the residual term shows that $T_F$ is
absolutely minimizing. In particular, $F$ is a global
$\Phi$-perimeter minimizer, not only a minimizer relative to $H^+$.
Since $F\subset H^+$ and its boundary contains $\Pi$, \eqref{wall}
gives $F=H^+$.

The complement $E^c$ is a global $\Phi^-$-perimeter minimizer.
Applying the same argument to $G=E^c\cap H^-$ gives $G=H^-$.
Thus $E$ fills $H^+$ and is empty in $H^-$, proving the assertion.
\end{proof}

\Needspace{8.0cm}
Figure~\ref{fig:thin-set} illustrates the connectivity step in
Lemma~\ref{planar-component}. For $K=\Pi\cap\sing S_E$ and
$a,b\in\Pi\setminus K$, an intermediate point $q$ can be chosen
so that $[a,q]\cup[q,b]\subset\Pi\setminus K$. Thus
$\Pi\setminus K$ is connected. The decisive hypothesis is
$\Hn^{n-1}(K)=0$, not merely zero $n$-dimensional measure.
The purple dots represent $K$ only schematically; no discreteness
assumption is made.
\begin{figure}[H]
\centering
\begin{tikzpicture}[x=1cm,y=1cm,font=\small]
\path[use as bounding box] (0,0) rectangle (15.65,3.65);
\filldraw[fill=gglight,draw=ggblue!55,line width=.55pt]
(.25,.5)--(8.45,.5)--(9.5,3.35)--(1.30,3.35)--cycle;
\foreach \x/\y in {1.4/1.0,2.4/.8,3.5/.85,4.6/.82,5.9/.9,7.3/1.0,8.2/1.4,1.65/2.9,2.7/3.05,3.7/2.9,4.8/3.1,6.1/2.95,7.2/3.1,8.2/2.8,2.45/2.18,3.3/2.6,4.1/2.2,5.0/2.42,5.8/2.3,6.9/2.6,7.8/1.9,3.6/1.38,6.55/1.35}
  \fill[ggpurple] (\x,\y) circle (1.4pt);
\coordinate (a) at (1.77,2.13);
\coordinate (q) at (4.8,1.43);
\coordinate (b) at (8.17,2.25);
\draw[ggline] (a)--(q)--(b);
\draw[ggarrow,ggblue] ($(a)!.36!(q)$)--($(a)!.55!(q)$);
\draw[ggarrow,ggblue] ($(q)!.42!(b)$)--($(q)!.61!(b)$);
\foreach \p in {a,q,b} {\fill[ggblue] (\p) circle (2pt);}
\node[below left,text=ggblue] at (a) {$a$};
\node[below,text=ggblue] at (q) {$q$};
\node[right,text=ggblue] at (b) {$b$};
\node[font=\large] at (.85,.81) {$\Pi$};
\node[anchor=west] at (10.05,2.38) {$K=\Pi\cap\sing S_E$};
\node[anchor=west] at (10.05,1.77) {$\Hn^{n-1}(K)=0$};
\node[anchor=west] at (10.05,1.10) {$[a,q]\cup[q,b]\subset\Pi\setminus K$};
\fill[ggpurple] (10.19,.42) circle (1.4pt);
\node[anchor=west,ggsmall] at (10.43,.42) {$K$};
\end{tikzpicture}
\space
\caption{Planar connectivity.}
\label{fig:thin-set}
\end{figure}
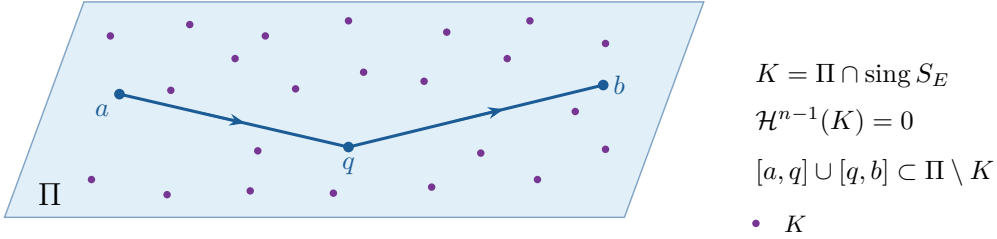

\section{Maximum sets under convergence}\label{sec-maximum}

In this section we prove the compactness statement at the heart of the
argument. A normalized subsolution whose upper bounds converge to one
cannot lose all of its maximum set at singular points. Applying the
mean-value estimate to each positive truncation produces a set of
positive measure on which the limit equals one.

\begin{proposition}[Maximum-set compactness]\label{maximum-compactness}
Suppose $n\geq 3$, and let $\Sigma_j=S_{E_j}$ be smooth global
$\Phi$-minimizing boundaries with $0\in\Sigma_j$. Let $L_j,Q_j$ be
as in \eqref{jacobi-operator}, and write $\mu_j=\Hn^n\restr\Sigma_j$.
Suppose $w_j\in W^{1,2}_{\mathrm{loc}}(\Sigma_j)\cap C^0(\Sigma_j)$ are
nonnegative and satisfy
\begin{equation}\label{normalized}
L_jw_j\geq Q_jw_j\quad\text{weakly on }\Sigma_j,\qquad
\max_{\Sigma_j\cap\overline B_1}w_j=1.
\end{equation}
Assume also that, for every fixed $R\geq 1$,
\begin{equation}\label{upper-bound}
\limsup_{j\to\infty}\sup_{\Sigma_j\cap\overline B_R}w_j\leq 1.
\end{equation}
Then a subsequence of $E_j$ converges locally to a halfspace, and the
corresponding boundaries converge smoothly on compact sets to its
boundary plane.
\end{proposition}

\begin{proof}
Use \eqref{compactness} to obtain a global minimizing limit $E$, and
write $S=S_E$ and $\mu=\mu_E$. On $\reg S$, let $L_S,Q_S,A_S$ be
the operator, potential, and second fundamental form defined by the
conventions following \eqref{jacobi-operator}. In particular, $0\in S$.
On a relatively compact regular chart of $S$, the boundaries $\Sigma_j$
are single smooth graphs for large $j$. Choose a nonnegative smooth
cutoff $\eta$ compactly supported in this chart and pull it back to the
approximating graphs, retaining the same symbol. Testing
\eqref{normalized} with $\eta^2w_j$ and using Cauchy--Schwarz gives
\begin{equation}\label{caccioppoli}
\int_{\Sigma_j}\eta^2|\nabla w_j|_{P_j}^2\,d\mu_j
+\int_{\Sigma_j}\eta^2Q_jw_j^2\,d\mu_j
\leq 4\int_{\Sigma_j}w_j^2|\nabla\eta|_{P_j}^2\,d\mu_j.
\end{equation}
Indeed, if the first two terms are denoted by $I$ and $J$, the test
gives
$I+J\leq 2I^{1/2}(\int w_j^2|\nabla\eta|_{P_j}^2)^{1/2}$;
absorption proves the displayed inequality.

The local bounds from \eqref{upper-bound} and \eqref{caccioppoli} give
uniform $W^{1,2}$ bounds after identification with a fixed regular
chart. By Rellich compactness and a diagonal extraction, there is
$w\in W^{1,2}_{\mathrm{loc}}(\reg S)$ to which the pullbacks of $w_j$
converge strongly in local $L^2$ and weakly in local $W^{1,2}$. Smooth
graphical convergence also gives local uniform convergence of the
coefficient matrices and curvature potentials. Passing to the weak
inequality on these charts yields
\begin{equation}\label{limit-inequality}
0\leq w\leq 1,\qquad L_Sw\geq Q_Sw\geq 0\quad\text{on }\reg S.
\end{equation}

We next justify convergence of the integrals needed for the mean-value
estimate, including neighborhoods of the singular set. Choose $\rho>1$
with $\mu(\partial B_\rho)=0$, and fix $k\in(0,1)$. We claim that
\begin{equation}\label{integral-convergence}
\int_{\Sigma_j\cap B_\rho}(w_j-k)_+^2\,d\mu_j
\longrightarrow
\int_{S\cap B_\rho}(w-k)_+^2\,d\mu.
\end{equation}
The function $w$ can be assigned any value on $\sing S$ in the
right-hand integral, since this set is $\mu$-null.

Here are the details of the claim. In a slightly larger closed ball,
the singular set is compact and has zero $\Hn^n$ measure. For any
$\varepsilon>0$ it can be covered by finitely many ambient balls
$B_{r_i}(Y_i)$ with $\sum_i r_i^n<\varepsilon$. The upper bound in
\eqref{density} gives
\[
\mu_j\left(\bigcup_i B_{r_i}(Y_i)\right)
+\mu\left(\bigcup_i B_{r_i}(Y_i)\right)\leq C\varepsilon.
\]
If a center is not on $\Sigma_j$, one chooses a point of $\Sigma_j$ in
the ball and doubles its radius; empty intersections contribute nothing.
The integrands in \eqref{integral-convergence} are uniformly bounded on
this fixed region, so the contribution of the cover is $O(\varepsilon)$.
The part of $S$ in the closed ball and outside the cover is compact and
regular. Finitely many graphical charts, with a partition of unity,
give convergence of the integrals there from strong local $L^2$
convergence. Support convergence excludes any remaining portion of
$\Sigma_j$ away from these charts and the cover. Finally, by
\eqref{compactness} and $\mu(\partial B_\rho)=0$, the upper limit as
$j\to\infty$ of the mass in a thin annulus about $\partial B_\rho$
tends to zero with its width. Sending first $j\to\infty$ and then
$\varepsilon\to 0$ proves \eqref{integral-convergence}.

Fix one such $\rho\in(4,5)$. The positive truncation $(w_j-k)_+$ is
a nonnegative subsolution of $L_j$. The normalization in
\eqref{normalized} and the mean-value estimate on $B_2\subset B_4$ give
\begin{equation}\label{truncated-mean-value}
1-k\leq C\left(\int_{\Sigma_j\cap B_\rho}(w_j-k)_+^2\,d\mu_j\right)^{1/2},
\end{equation}
where $C=C(n,\Phi)$ is independent of $j$ and $k$. At each fixed $k$,
pass to the limit using \eqref{integral-convergence}. Since $w\leq 1$,
\[
1-k\leq C\left(\int_{S\cap B_\rho}(w-k)_+^2\,d\mu\right)^{1/2}
\leq C(1-k)\,\mu\bigl(S\cap B_\rho\cap\{w>k\}\bigr)^{1/2}.
\]
Cancel $1-k$ and then let $k\uparrow 1$. Continuity from above for the
finite measure $\mu\restr B_\rho$ gives
\begin{equation}\label{positive-maximum-set}
\mu\bigl(S\cap B_\rho\cap\{w=1\}\bigr)\geq C^{-2}>0.
\end{equation}
Thus a regular chart meets the maximum set in positive measure.

On that chart, $h=1-w$ is a nonnegative weak supersolution of
$L_Sh\leq 0$. In local coordinates $x$, choose a Lebesgue density point
$y$ of $\{h=0\}$ and a radius $r>0$ so small that
$\overline B_{2r}^n(y)$ lies in the chart and
$|\{h=0\}\cap B_r^n(y)|>0$. Coordinate volume and the induced measure
are locally equivalent. In particular,
$\essinf_{B_r^n(y)}h=0$. The local weak Harnack inequality
\cite[Theorem 8.18]{GT} gives, for some $p>0$,
\[
\left(\frac{1}{|B_r^n(y)|}\int_{B_r^n(y)}h^p\,dx\right)^{1/p}
\leq C\,\essinf_{B_r^n(y)}h=0.
\]
Hence $h=0$ almost everywhere on this ball. A finite chain of overlapping
coordinate balls along any path in the same regular component propagates
this equality. Thus $w=1$ almost everywhere on a connected component
$U$ of $\reg S$. No uniform Harnack constant across the singular set
is needed here.

From \eqref{limit-inequality}, $0\geq Q_S$ on $U$; hence $Q_S=0$ and
$A_S=0$ there. Lemma~\ref{planar-component} implies that $E$ is a
halfspace. The regular convergence in Section~\ref{sec-estimates} now
applies on every compact subset of its boundary plane, proving the last
assertion.
\end{proof}

\section{The growth gap}\label{sec-growth}

In this section we turn the compactness statement into a comparison
between two fixed scales. Iteration gives a power-growth lower bound at
every sufficiently large radius for a nonaffine graph. The scale and the
exponent depend only on the fixed integrand, although the compactness
argument does not provide numerical values for them.

For an entire graph $\Sigma$, $O\in\Sigma$, and $r>0$, define
\begin{equation}\label{V-definition}
V_O(r)=\max_{\Sigma\cap\overline B_r(O)}v,
\qquad v=(\nu\cdot e_{n+1})^{-1}.
\end{equation}
The graph is proper, so the maximum is finite and attained. The function
$V_O$ is nondecreasing.

\begin{lemma}[Jacobi comparison]\label{jacobi-comparison}
Suppose $\Sigma_j$ are upward entire $\Phi$-minimal graphs through zero,
with upward unit normals $\nu_j$ and
$v_j=(\nu_j\cdot e_{n+1})^{-1}$, and their subgraphs converge locally
to a halfspace. If the boundary
convergence is smooth on compact subsets of the limiting plane, then
for all sufficiently large $j$,
\begin{equation}\label{harnack-comparison}
\max_{\Sigma_j\cap\overline B_1}v_j\leq Cv_j(0),
\end{equation}
where $C$ depends only on $n,\Phi$ and not on the angle of the limiting
plane with the vertical direction.
\end{lemma}

\begin{proof}
Let $N$ be the outer unit normal of the limiting halfspace. Since the
limiting plane passes through zero, it is $N^\perp$. On a fixed large compact portion
of this plane, smooth one-sheet convergence gives a
single small-slope graph containing $\Sigma_j\cap\overline B_1$ and
zero, with a fixed interior margin. Define
\[
h_j=\nu_j\cdot N,\qquad f_j=\nu_j\cdot e_{n+1}.
\]
For large $j$, $1/2\leq h_j\leq 1$, whereas $f_j>0$ since $\Sigma_j$
is an upward graph. Both functions solve the translation Jacobi equation
\eqref{translation-jacobi}. Subtracting the two equations after
multiplication by the other field gives
\begin{equation}\label{quotient-equation}
\Div_{\Sigma_j}\bigl(h_j^2P_j\nabla_{\Sigma_j}(f_j/h_j)\bigr)=0.
\end{equation}
In coordinates on $N^\perp$ this is a divergence-form equation with a
uniformly elliptic coefficient matrix. Its ellipticity constants are
controlled by \eqref{ellipticity}, $h_j\geq 1/2$, and the small slope
of the chart. Interior Harnack for the positive quotient gives
\[
\frac{f_j(X)}{h_j(X)}\geq c\,\frac{f_j(0)}{h_j(0)}
\qquad(X\in\Sigma_j\cap\overline B_1).
\]
Since $h_j$ is bounded above and below, this proves
\eqref{harnack-comparison}. No lower bound for $N\cdot e_{n+1}$ has
been used.
\end{proof}

\Needspace{10.0cm}
Figure~\ref{fig:jacobi} illustrates Lemma~\ref{jacobi-comparison}.
The same schematic local curve is shown in ambient coordinates in
panel (a) and in coordinates adapted to $N^\perp$ in panel (b).
Since $\nu_j$ stays close to $N$, we have
$h_j=\nu_j\cdot N\geq1/2$ for large $j$, even when the vertical
component $f_j=\nu_j\cdot e_{n+1}$ is small. Harnack applied to the
positive quotient $f_j/h_j$ yields \eqref{harnack-comparison} for
$v_j=1/f_j$. The comparison controls relative, not absolute, steepness;
no positive lower bound on $N\cdot e_{n+1}$ is used.
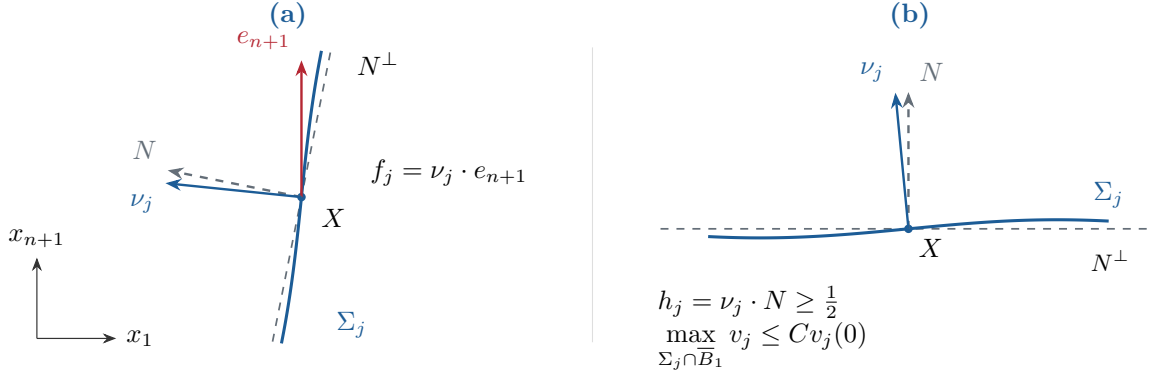
\begin{figure}[H]
\centering
\begin{tikzpicture}[x=1cm,y=1cm,font=\small]
\path[use as bounding box] (0,0) rectangle (15.65,5.1);
\node[font=\small\bfseries,text=ggblue] at (3.72,4.83) {(a)};
\node[font=\small\bfseries,text=ggblue] at (11.95,4.83) {(b)};
\draw[black!13] (7.75,.5)--(7.75,4.4);
\coordinate (Xl) at (3.9,2.42);
\draw[gggray,dashed,line width=.65pt]
({3.9-.2*1.95/sqrt(1.04)},{2.42-1.95/sqrt(1.04)})--
({3.9+.2*1.95/sqrt(1.04)},{2.42+1.95/sqrt(1.04)});
\draw[ggline] plot[domain=-1.95:1.95,samples=160]
({3.9+(.2*\x-.12*sin(deg(.8*\x)))/sqrt(1.04)},
 {2.42+(\x+.024*sin(deg(.8*\x)))/sqrt(1.04)});
\fill[ggblue] (Xl) circle (1.6pt);
\node[right] at (4.03,2.14) {$X$};
\draw[ggarrow,ggred] (Xl)--++(0,1.81) node[above left,text=ggred] {$e_{n+1}$};
\draw[ggarrow,gggray,dashed] (Xl)--++(-1.77,.354) node[above left,text=gggray] {$N$};
\draw[ggarrow,ggblue] (Xl)--++(-1.80,.184) node[below left,text=ggblue] {$\nu_j$};
\node[anchor=west] at (4.50,4.19) {$N^\perp$};
\node[anchor=west,text=ggblue] at (4.23,.73) {$\Sigma_j$};
\draw[ggaxis] (.40,.55)--(1.45,.55) node[right] {$x_1$};
\draw[ggaxis] (.40,.55)--(.40,1.62) node[above] {$x_{n+1}$};
\node[anchor=west] at (4.68,2.75) {$f_j=\nu_j\cdot e_{n+1}$};
\coordinate (Xr) at (11.93,2.00);
\draw[gggray,dashed,line width=.65pt] (8.65,2)--(15.18,2);
\draw[ggline] plot[domain=-2.65:2.65,samples=160]
({11.93+\x},{2+.12*sin(deg(.8*\x))});
\fill[ggblue] (Xr) circle (1.6pt);
\draw[ggarrow,gggray,dashed] (Xr)--++(0,1.81) node[above right,text=gggray] {$N$};
\draw[ggarrow,ggblue] (Xr)--++(-.173,1.802) node[above left,text=ggblue] {$\nu_j$};
\node[below right] at (Xr) {$X$};
\node[below,ggsmall] at (14.60,1.86) {$N^\perp$};
\node[above,text=ggblue] at (14.58,2.14) {$\Sigma_j$};
\node[anchor=west] at (8.48,1.03) {$h_j=\nu_j\cdot N\geq\tfrac12$};
\node[anchor=west] at (8.48,.43) {$\displaystyle\max_{\Sigma_j\cap\overline B_1}v_j\leq C v_j(0)$};
\end{tikzpicture}
\space
\caption{Jacobi comparison.}
\label{fig:jacobi}
\end{figure}

\begin{proposition}[Fixed-scale growth]\label{fixed-scale}
Suppose $n\geq 3$. There are constants $K>1$ and $L>2$, depending
only on $n,\Phi$, such that every entire $\Phi$-minimal graph satisfies
\begin{equation}\label{doubling-comparison}
V_O(r)\leq\max\left\{Kv(O),\frac12 V_O(Lr)\right\}
\qquad(O\in\Sigma,\ r>0).
\end{equation}
\end{proposition}

\begin{proof}
We first prove that there are $A>1$, $C_0>1$, and $\theta\in(0,1)$,
depending only on $n,\Phi$, for which
\begin{equation}\label{initial-comparison}
V_O(r)\leq\max\{C_0v(O),\theta V_O(Ar)\}
\end{equation}
holds for all graphs, centers, and radii. If no such triple exists, take
sequences $A_j,C_j>1$ and $\theta_j\in(0,1)$ such that
$A_j\to\infty$, $C_j\to\infty$, and $\theta_j\uparrow 1$. For each
$j$ there are a graph $\widehat\Sigma_j$, a point
$O_j\in\widehat\Sigma_j$, and $r_j>0$ for which the corresponding
comparison fails. Let $\widehat v_j$ be the reciprocal vertical Jacobi
field on $\widehat\Sigma_j$ and, for $s>0$, set
\[
\widehat V_{O_j}(s)=\max_{\widehat\Sigma_j\cap\overline B_s(O_j)}
\widehat v_j.
\]
Translate by $-O_j$, dilate by $r_j^{-1}$, and normalize the field:
\[
\Sigma_j=r_j^{-1}(\widehat\Sigma_j-O_j),\qquad
w_j(X)=\frac{\widehat v_j(O_j+r_jX)}{\widehat V_{O_j}(r_j)}
\quad(X\in\Sigma_j).
\]
The resulting function satisfies
\[
L_jw_j\geq Q_jw_j,\qquad
\max_{\Sigma_j\cap\overline B_1}w_j=1,\qquad
w_j(0)<C_j^{-1},
\]
and
\[
0<w_j<\theta_j^{-1}\qquad\text{on }\Sigma_j\cap\overline B_{A_j}.
\]
Thus Proposition~\ref{maximum-compactness} applies. After a subsequence,
the subgraphs converge to a halfspace and the graphs converge smoothly
to its plane. By Lemma~\ref{jacobi-comparison},
\[
1=\max_{\Sigma_j\cap\overline B_1}w_j
\leq Cw_j(0)<C/C_j,
\]
a contradiction.

Choose an integer $m\geq 1$ such that $\theta^m\leq 1/2$. Iteration
of \eqref{initial-comparison} gives
\[
V_O(r)\leq\max\{C_0v(O),\theta^mV_O(A^mr)\}.
\]
Take $K=C_0$ and enlarge $A^m$, if necessary, to a number $L>2$.
Monotonicity of $V_O$ proves \eqref{doubling-comparison}.
\end{proof}

Put
\begin{equation}\label{beta}
\beta=\frac{\log 2}{\log L}\in(0,1).
\end{equation}
The fixed-scale comparison has two useful consequences. First, for
$R\geq r>0$, let $\ell=\lfloor\log_L(R/r)\rfloor$. Iteration up to
$\ell$ gives
\begin{equation}\label{interpolation}
V_O(r)\leq\max\left\{Kv(O),\,2\left(\frac rR\right)^\beta V_O(R)\right\}.
\end{equation}
Indeed, $L^\ell r\leq R$ and $2^{-\ell}\leq 2(r/R)^\beta$.
Second, once $V_O(r_0)>Kv(O)$, monotonicity keeps this inequality valid
at all larger radii, and \eqref{doubling-comparison} forces doubling at
every step. Hence
\begin{equation}\label{all-radii-growth}
V_O(R)\geq\frac12 V_O(r_0)\left(\frac R{r_0}\right)^\beta
\qquad(R\geq r_0).
\end{equation}
These are quantitative comparisons in terms of $K,L$. Their existence,
rather than effective bounds for their values, is what
Proposition~\ref{fixed-scale} establishes.

\begin{proof}[Proof of Theorem~\ref{main}]
For $n=2$ the assertion follows from Jenkins' theorem \cite{Jenkins};
one may take, for example, $\alpha=1/2$. Assume $n\geq 3$ and choose
$\alpha=\beta$ from \eqref{beta}. Let $O=(0,u(0))$. Projection of an
ambient ball to the base gives
\begin{equation}\label{projection-bound}
V_O(R)\leq\left(1+\left(\sup_{B_R^n}|Du|\right)^2\right)^{1/2}.
\end{equation}
Here the supremum over the open base ball equals that over its closure,
by continuity of $Du$. Under \eqref{small-growth}, $V_O(R)=o(R^\beta)$.
Fix $r$ and let $R\to\infty$ in \eqref{interpolation}. It follows that
$V_O(r)\leq Kv(O)$ for every $r>0$, so $Du$ is bounded on $\R^n$.

For completeness, write $D^2\varphi=(\varphi_{ij})$ and differentiate
\eqref{graph-equation}. Each $u_k=\partial_k u$ solves
\[
\partial_i\bigl(\varphi_{ij}(Du)\,\partial_j u_k\bigr)=0
\qquad\text{in }\R^n.
\]
The coefficient matrix is uniformly elliptic because $Du$ is bounded
and $D^2\varphi$ is positive definite on compact sets. The interior
oscillation estimate for bounded solutions, applied on $B_R^n$ and then
with $R\to\infty$, makes every $u_k$ constant \cite[Chapter 8]{GT}.
Thus $u$ is affine.

If instead $u$ is nonaffine, this same Liouville argument shows that
$V_O$ is unbounded. Choose $r_0$ so that $V_O(r_0)>Kv(O)$.
Combining \eqref{all-radii-growth} and \eqref{projection-bound}, and
increasing $r_0$ if needed, gives constants $c>0$ and $R_0<\infty$,
depending on the solution, such that
\[
\sup_{B_R^n}|Du|\geq cR^\beta\qquad(R\geq R_0).
\]
This proves the additional growth assertion.
\end{proof}

\begin{remark}
The dependence of the exponent in \eqref{beta} is implicit. To obtain
an explicit lower bound for $\beta$ from specified bounds on $\Phi$,
one would need an effective version of the compactness step leading to
\eqref{initial-comparison}. In particular, the present argument supplies
no quantitative rate at which almost-maximal subsolutions force
closeness to a plane. Nor does it identify the optimal growth exponent
or its behavior as $\Phi$ approaches the area integrand.
\end{remark}

\end{document}